\documentclass[11pt]{article}

\usepackage{amsfonts}
\usepackage{amsmath,bbm}
\usepackage{color, diagrams}

\newcommand{\beq}{\begin{equation}}
\newcommand{\eeq}{\end{equation}}
\newcommand{\bea}{\begin{eqnarray}}
\newcommand{\eea}{\end{eqnarray}}
\newcommand{\beas}{\begin{eqnarray*}}
\newcommand{\eeas}{\end{eqnarray*}}

\newtheorem{theorem}{Theorem}[section]

\newtheorem{corollary}[theorem]{Corollary}

\newtheorem{remark}[theorem]{Remark}
\newtheorem{example}[theorem]{Example}
\newtheorem{examples}[theorem]{Examples}
\newtheorem{foo}[theorem]{Remarks}

\newenvironment{proof}{\addvspace{\medskipamount}\par\noindent{\it
Proof}.}
{\unskip\nobreak\hfill$\Box$\par\addvspace{\medskipamount}}

\newcommand{\edg}[1]{\left[#1\right]}     

\newcommand{\bS}{\mathbb S}

\newcommand{\R}{\mathbb R}

\newcommand{\fH}{\mathbf{H}^{2n+1}}
\newcommand{\bsn}{\mathbb{S}^{2n+1}}

\newcommand{\bE}{\mathbb E}
\newcommand{\C}{\mathbb C}

\title{Conformal transforms and Doob's $h$-processes on Heisenberg groups}
\author{Jing Wang\footnote{wangjing@illinois.edu} }

\date{Department of Mathematics, \\
University of Illinois at Urbana Champaign \\
 Urbana, IL, USA}

\begin{document}
\maketitle

\begin{abstract}
We study the stochastic processes that are images of Brownian motions on Heisenberg group $\fH$ under conformal maps. In particular, we obtain that Cayley transform maps Brownian paths in $\fH$ to a time changed Brownian motion on CR sphere $\bsn$ conditioned to be at its south pole at a random time. We also obtain that the inversion of Brownian motion on $\fH$ started from $x\not=0$, is up to time change, a Brownian bridge on $\fH$ conditioned to be at the origin. 
\end{abstract}

\tableofcontents
\clearpage

\section{Introduction}
The Brownian motions on sub-Riemannian model spaces has been widely studied in recent years. Due to strong symmetries of  the model spaces, explicit computations analysis can be conducted (see \cite{Ga}, \cite{BGG}, \cite{BB}, \cite{BW}, and \cite{BW1}). In this paper we focus on the relationships between Brownian motion on Heisenberg group and its images under certain conformal maps, namely Cayley transform and Kelvin transform.  

Let $\fH$ be a $2n+1$ dimensional Heisenberg group that lives in $\C^n\times\R$ with coordinates $(z,t)=(z_1,\dots, z_n, t)$ where $z_j=x_j+iy_j$. It has the group law
\[
(z,t)(z',t')=(z+z',t+t'+\mathbf{Im}z\bar{z'}).
\]
It is a flat model space of sub-Riemannian manifolds. There is a canonical sub-Laplacian on $\fH$:
\[
\bar{L}_{\fH}=\sum_{j=1}^n\left(\frac{\partial^2}{\partial x_j^2}+\frac{\partial^2}{\partial y_j^2} + 2y_j\frac{\partial^2}{\partial x_j\partial t}-2x_j\frac{\partial^2}{\partial y_j\partial t}+|z_j|^2\frac{\partial^2}{\partial t^2}\right)
\] 
The Brownian motion on $\fH$ issued from $x'\in \fH$ is the strong Markov process that is generated  by $\frac{1}{2}\bar{L}_{\fH}$. 

Cayley transform is known to be a bi-holomorphic map between the Siegel domain $\Omega^{n+1}$ and a unit ball in $\C^{n+1}$. The restriction of Cayley transform on its boundary therefore provides a conformal map between $\fH$ and the unit sphere $\bsn$ in $\C^{n+1}$. If we consider the image of a Brownian path on $\fH$ under Cayley transform, it then turns out to be a $\bsn$-valued process. In particular, it is a time changed version of a Brownian path on $\bsn$ conditioned to be at the south pole at a random time. Below we state our main result.
\begin{theorem}\label{thm-bsn}
The  Brownian motion on $\fH$ issued from $x'$ is mapped by Cayley transform $\mathcal{C}_1$ to a time-changed Brownian motion on $\bsn$ issued from $x=\mathcal{C}_1(x')$ and conditioned to be at the south pole $-e_n$ at time $T$, 
where $T$ is an independent random variable with distribution
\begin{equation}\label{dis-T}
\mathbb{P}^h_x\edg{T>t}=\frac{\int_t^{+\infty}e^{-n^2s}p_s(-e_n,x)ds}{\int_0^{+\infty}e^{-n^2t}p_t(-e_n,x)dt}.
\end{equation}
Here $p_t(x,y)$ denotes the subelliptic heat kernel on $\bsn$.
\end{theorem}
This result extends the result by Carne in  \cite{Carne}, where he proved that the Stereographic projection from $\R^n$ to $S^n$ maps Brownian paths in $\R^n$ to the paths of conditioned Brownian motion on $S^n$.

Another object of our study is to probabilistically interpret  the relation between the Brownian motion on $\fH$ started from any $x'\not=0$ and its image under the inversion map, namely the Kelvin transform. This type of question was first posed by Schwartz (see \cite{Sch}), who asked how Brownian motion in $\R^n$ can be interpreted as a Brownian bridge conditioned to be at the ``ideal point at infinity". A probabilistic approach was provided by Yor in \cite{Yor}. In the present paper, we obtain the result in a setting of a flat sub-Riemannian manifold. The inversion of Brownian motion on $\fH$ issued from $x\not=0$ turns out to be  a Brownian bridge conditioned to be at the origin up to time change.
\begin{theorem}\label{thm-K}
The  Brownian motion on $\fH$ generated by $\frac{1}{2}L_{\fH}$ and issued from $x'\not=0$ is mapped by Kelvin transform to a time-changed $\fH$-valued Brownian motion conditioned to be at the origin at $t=\infty$.
\end{theorem}

The approaches to both results follow the idea of Carne. By analyzing the radial part of the corresponding conformal sub-Laplacians on $\bsn$ and on $\fH$, we are able to obtain the relationship between Markov processes that are generated by $\frac{1}{2}L_{\bsn}$ and $\frac{1}{2}L_{\fH}$ respectively through an argument of Doob's $h$-processes. 

In the next section, we deduce Theorem \ref{thm-bsn} after a detailed discussion of Cayley transform and radial process or Brownian motions on $\bsn$ and $\fH$. In section 3 we focus on the inverse transform on $\fH$ and the proof of  Theorem \ref{thm-K}.
\section{Cayley transformation and Doob's $h$-process}
\subsection{Cayley transform on CR model spaces}
Cayley transforms on CR model spaces are natural analogues of stereographic projections on Riemannian models. Let $B^{n+1}=\{\zeta\in \C^{n+1}:|\zeta|<1\}$ be the unit ball in $\C^{n+1}$ and  $\Omega^{2n+1}=\{(z,w)\in \C^n\times \C, \mathbf{Im}(w)>|z|^2 \}$ the Siegel domain. The Cayley transform $\mathcal{C}: B^{2n+1}\to\Omega^{n+1}$ is a biholomorphic map such that (see \cite{CR})
\[
\mathcal{C}: \left(\zeta_1,\dots, \zeta_{n+1}\right)\to\left( \frac{\zeta_1}{1+\zeta_{n+1}},\dots,  \frac{\zeta_n}{1+\zeta_{n+1}}, i \frac{1-\zeta_{n+1}}{1+\zeta_{n+1}}\right), \quad \zeta^{n+1}\not=-1.
\]
Let $\bsn=\{\zeta\in \C^{n+1}, |\zeta|=1\}$ be the unit sphere in $\C^{n+1}$. It also appears as a model space of CR manifolds. The restriction of $\mathcal{C}$ to the CR sphere $\bsn$ minus a point gives a CR diffeomorphism to the boundary of the Siegel domain $\partial\Omega^{2n+1}$, which may be identified with the Heisenberg group $\fH$ through the CR isomorphism $\varphi:\fH\to \partial \Omega^{2n+1}$. For any $(z,t)\in \fH$,
\begin{equation}\label{eq-varphi}
\varphi(z,t)=(z,2t+i|z|^2).
\end{equation}
We denote the north pole of $\bS^{2n+1}$  by $e_n=\{0,\dots, 0,1\}$ and corresponding the south pole by $-e_n$. Now we consider the CR equivalence between Heisenberg group and CR sphere minus the south pole  $\mathcal{C}_1:\fH\to \bsn\setminus\{ -e_n\}$. It is then given by $\mathcal{C}_1= \mathcal{C}^{-1}\circ\varphi$. In local coordinates we have for any $\left( z,t\right)=(z_1,\dots, z_n, t)\in\fH$,
\begin{equation}\label{eq-C1-trans}
\mathcal{C}_1: \left( z,t\right)\to\left(\frac{2z_1}{(1+|z|^2)-2it},\dots,  \frac{2z_n}{(1+|z|^2)-2it}, \frac{1-|z|^2+2it}{1+|z|^2-2it}\right).
\end{equation}
It is a conformal map with inverse $\mathcal{C}_1^{-1}: \bsn\setminus\{ -e_n\}\to \fH$,
\begin{equation}\label{Cayley}
\mathcal{C}_1^{-1}: (\zeta_1,\cdots \zeta_{n+1})\to \left( \frac{\zeta_1}{1+\zeta_{n+1}},\dots,  \frac{\zeta_n}{1+\zeta_{n+1}}, \frac{i}{2}\frac{\overline{\zeta_{n+1}}-\zeta_{n+1}}{|1+\zeta_{n+1}|^2}\right).
\end{equation}
Since $\bsn$ is a model space of sub-Riemannian manifold with the Hopf fibration $\bS^1\to \bsn\to\mathbb{CP}^n$, it is more convenient for us to use the so-called cylindrical coordinates that carries the structural information and are given by
\begin{align*}
(w,\theta)\to \frac{e^{i\theta} }{\sqrt{1+|w|^2}} \left(w, 1\right),
\end{align*}
where  $\theta \in \R/2\pi\mathbb{Z}$, and $w=\zeta/\zeta_{n+1} \in \mathbb{CP}^n$. Here $w=(w_1, \cdots, w_n)$ parametrizes the complex lines passing through the origin, and $\theta$ determines a point on the line that is of unit distance from the north pole. Let $|w|=\tan r_S$, $r_S\in[0,\pi/2)$, then we have $\mathcal{C}_1^{-1}$ in cylindrical coordinates  given by
\begin{align*}
\mathcal{C}_1^{-1}: \left(\frac{e^{i\theta} }{\sqrt{1+|w|^2}} \left(w, 1\right)\right)\to\left( \frac{e^{i\theta}\cos r_S+\cos^2r_S}{1+\cos^2r_S+2\cos r_S\cos\theta} w ,\frac{\cos r_S\sin\theta}{1+\cos^2r_S+2\cos r_S\cos\theta}\right).
\end{align*}
Let $\psi_S: \bS^{2n+1}\to [0,\pi/2)\times \R/2\pi \mathbb{Z}$ be such that
\[
\psi_S\left(\frac{e^{i\theta} }{\sqrt{1+|w|^2}} \left(w, 1\right)\right)=(r_S,\theta)
\]
and $\psi_H: \fH\to \R_{\ge0}\times \R$ be such that
\[
\psi_H\left(z,t\right)=(r_H,t),
\]
where $r_H=\sqrt{\sum_{j=1}^n|z_j|^2}$. We define a map $\R_{\ge0}\times \R\to [0,\pi/2)\times \R/2\pi \mathbb{Z}$ by the chart below, and by abusing of notation we denote it by $\mathcal{C}_1$:
\begin{diagram}\label{diag}
\fH & \rTo^{\mathcal{C}_1} &\bsn  \\
 \dTo^{\psi_H} & &\dTo_{\psi_S} \\
\R_{\ge0}\times \R& \rTo^{\mathcal{C}_1}  & [0,\pi/2)\times \R/2\pi \mathbb{Z} 
\end{diagram}
We easily compute that
\[
\mathcal{C}_1:(r_H,t)\to\left(\arcsin\left(\frac{2r_H}{\sqrt{(1+r_H^2)^2+4t^2}}\right),\arcsin\left(\frac{4t}{\sqrt{(1+r_H^2)^2+4t^2}\sqrt{(1-r_H^2)^2+4t^2}}\right)\right)
\]
and  
\[
\mathcal{C}_1^{-1}:\left(r_S, \theta\right)\to\left(\frac{\sin r_S}{\sqrt{1+\cos^2r_S+2\cos r_S\cos\theta}},\frac{\cos r_S\sin\theta}{1+\cos^2r_S+2\cos r_S\cos\theta} \right).
\]

\subsection{Brownian motion and Doob's $h$-process}
Now we consider the Markov processes that are generated by sub-Laplacians $\bar{L}_{\fH}$ and $\bar{L}_{\bsn}$, which are referred to as Brownian motions on $\fH$ and $\bsn$ respectively throughout this paper.  Due to the radial symmetries of these diffusion processes, it is sufficient for us to consider only  the radial part of the sub-Laplacians.

We denote by $L_{\fH}$ the radial part of the sub-Laplacian on $\fH$ in coordinates ($r_H$, $t$), it is defined on the space $D_H=\{f\in C^\infty(\R_{\ge0}\times\R,\R), \frac{\partial f}{\partial r_H}|_{r_H = 0}=0\}$. Let $L_{\bsn}$ be the radial part of $\bar{L}_{\bsn}$ in cylindric coordinates ($r_S$, $\theta$), with domain $D_S=\{f\in C^\infty([0,\frac{\pi}{2})\times \R/2\pi\mathbb{Z}, \R), \frac{\partial f}{\partial r_S}|_{r_S=0}=0 \}$. Then for any $f\in D_H$ and $g\in D_S$, we have 
\[
\bar{L}_{\fH}(f\circ \psi_H)=(L_{\fH} f)\circ \psi_H, \quad \bar{L}_{\bsn}(g\circ \psi_S)=(L_{\bsn} g)\circ \psi_S.
\]
It is known that $L_{\bsn}$ is essentially self-adjoint with respect to the volume measure $d\mu_{\bsn}=\frac{2\pi^n}{\Gamma(n)}(\sin r_S)^{2n-1}\cos r_S dr_Sd\theta$ on $\bsn$, and $L_{\fH}$ is essentially self-adjoint with respect to the volume measure $d\mu_{\fH}=\frac{2\pi^n}{\Gamma(n)} r_H^{2n-1}dr_Hdt$ on $\fH$. Moreover, we have explicitly
\begin{equation}\label{LH}
L_{\fH}=\frac{\partial^2}{\partial r_H^2}+\frac{2n-1}{r_H}\frac{\partial}{\partial r_H}+r_H^2\frac{\partial^2 }{\partial t^2}
\end{equation}
and (see \cite{BW}, \cite{BB}, also \cite{G})
\begin{equation}\label{LS}
L_{\bsn}=\frac{\partial^2}{\partial r_S^2}+((2n-1)\cot r_S-\tan r_S)\frac{\partial}{\partial r_S}+\tan^2r_S\frac{\partial^2}{\partial \theta^2}.
\end{equation}
Let us consider Green function of the conformal sub-Laplacian $-L_{\bsn}+n^2$ with pole $(0,0)$ (the north pole of $\bsn$) 
and denote it by $G_{\bsn}$. From  \cite{BW} we have 
\begin{equation}\label{Green-S}
G_{\bsn}((0,0),(r_S, \theta))=\frac{\Gamma\left(\frac{n}{2}\right)^2}{8\pi^{n+1}(1-2\cos r_S\cos \theta+\cos^2r_S)^{n/2}}.
\end{equation}
On the other hand the Green function of $-L_{\fH}$ with respect to $d\mu_{\fH}$ is given by
\begin{equation}\label{Green-H}
G_{\fH}((0,0),(r_H,t))=\frac{\Gamma\left(\frac{n}{2}\right)^2}{8\pi^{n+1}(r_H^4+4t^2)^{n/2}}
\end{equation}
We consider $h\in D_S$, such that for any  $(r_S, \theta)\in [0,\frac{\pi}{2})\times \R/2\pi\mathbb{Z}$,
\begin{equation}\label{eq-h}
h(r_S, \theta)=1+2\cos r_S\cos\theta+\cos^2r_S,
\end{equation}
and $H\in D_H$, such that for any $(r_H,t)\in \R_{\ge0}\times\R$,
\begin{equation}\label{eq-H-har}
H(r_H, t)=\frac{4}{(1+r_H^2)^2+4t^2}.
\end{equation}
It is an easy fact that $h$ and $H$ are harmonic functions with poles $(0,\pi)$ and $(0,0)$ respectively. Moreover, we have
\[
H=\mathcal{C}_1^*h=h\circ \mathcal{C}_1.
\]
From \eqref{Green-S} and \eqref{Green-H} we can easily observe that
\[
G_{\bsn}((0,0),(r_S, \theta))(1+2\cos r_S\cos\theta+\cos^2r_S)^{\frac{n}{2}}=(\mathcal{C}_1^{-1*}G_{\fH})((0,0),(r_S, \theta)).
\]
In fact, for any $x,y \in [0,\frac{\pi}{2})\times \R/2\pi\mathbb{Z}$ we have 
\begin{equation}\label{Green-relation}
G_{\bsn}(x,y)=(\mathcal{C}_1^{-1*}G_{\fH})(x,y)h(x)^{-\frac{n}{2}}h(y)^{-\frac{n}{2}}.
\end{equation}
From this we can then deduce the relation between $L_{\fH}$ and $L_{\bsn}-n^2$.
\begin{theorem}
For any function $f\in D_S$, the relation of $L_{\fH}$ and $L_{\bsn}-n^2$ via Cayley transform is given by
\begin{equation}\label{LH-LS}
h^{(\frac{n}{2}+1)} (-L_{\bsn}+n^2) \left( h^{-\frac{n}{2}}f\right)=-(\mathcal{C}_{1*}L_{\fH})f
\end{equation}
where $h$ is as in \eqref{eq-h}.
\end{theorem}
\begin{proof}
For any $f\in D_S$,  let $F\in D_H$ be  such that  $F=(\mathcal{C}_1)^*f=f\circ\mathcal{C}_1$. We assume for some $\sigma_1, \sigma_2\in D_S$ it holds that for any $x\in [0,\frac{\pi}{2})\times \R/2\pi\mathbb{Z}$,
\[
 (-L_{\bsn}+n^2) \left( \sigma_1f \right)|_x
 =-\sigma_2\left(L_{\fH}\right)(\mathcal{C}_{1}^*f)|_{\mathcal{C}_1^{-1}(x)}.
\]
It then amounts to find $\sigma_1, \sigma_2$. Let $g=-L_{\fH}F$, then $F=(-L_{\fH})^{-1}g$. The above equation is equivalent to
\begin{equation}
\sigma_1\cdot\left((-L_{\fH})^{-1}g\right)\circ\mathcal{C}_1^{-1}=(-L_{\bsn}+n^2)^{-1}(\sigma_2(g\circ\mathcal{C}_1^{-1})).
\end{equation}
Therefore, for all $x\in[0,\frac{\pi}{2})\times \R/2\pi\mathbb{Z}$, we have
\begin{equation}\label{Green-R}
\int G_{\fH}(\mathcal{C}_1^{-1}(x),v)g(v)d\mu_{\fH}v=\sigma_1^{-1}(x)\int G_{\bsn}(x,y)\sigma_2(y)g(\mathcal{C}_1^{-1}(y))d\mu_{\bsn} y
\end{equation}
where $G_{\bsn}$ and $G_{\fH}$ are Green functions as in \eqref{Green-S} and \eqref{Green-H}.
Moreover by changing variable $y=\mathcal{C}_1(v)$, the right hand side of the above equation writes 
\begin{equation}\label{Green-Relation}
\sigma_1^{-1}(x)\int G_{\bsn}(x,\mathcal{C}_1(v))\sigma_2(\mathcal{C}_1(v))g(v)|J_{\mathcal{C}_1}(v)|d\mu_{\fH}v,
\end{equation}
where $|J_{\mathcal{C}_1}(v)|$ is the Jacobi determinant. We can easily compute that 
\[
|J_{\mathcal{C}_1}(v)|=H^{n+1}(v),
\]
where $H$ is given as in \eqref{eq-H-har}.
Therefore \eqref{Green-Relation} becomes
\[
\sigma_1^{-1}(x)\int G_{\bsn}(x,\mathcal{C}_1(v))\sigma_2(\mathcal{C}_1(v))g(v)H^{n+1}(v)d\mu_{\fH}v.
\]
By plugging in \eqref{Green-relation} and comparing to \eqref{Green-R}, we obtain for all $x, y\in\bsn $
\[
\begin{cases}
\sigma_1(x)
=h^{-\frac{n}{2}}(x)
\\
\sigma_2(y)
=h^{-(1+\frac{n}{2})}(y),
\end{cases}
\] 
hence the conclusion.
\end{proof}

\begin{corollary}
For any function $f\in D_S$, we have that
\begin{equation}\label{Doob}
(\mathcal{C}_{1*}L_{\fH})f=h \left(L_{\bsn} f+\frac{2\Gamma_{\bsn}(h^{-\frac{n}{2}},f)}{h^{-\frac{n}{2}}}\right)
\end{equation} 
where $\Gamma_{\bsn}(f,g)=\frac{1}{2}(L_{\bsn}(fg)-fL_{\bsn}g-gL_{\bsn}f)$ for any $f,g\in  D_S$.
\end{corollary}
\begin{proof}
Notice that 
\[
(L_{\bsn}-n^2)(h^{-\frac{n}{2}})=0.
\]
hence
\[
h^{\frac{n}{2}}(L_{\bsn}-n^2)(h^{-\frac{n}{2}}f)=L_{\bsn}f+2h^{\frac{n}{2}}\Gamma_{\bsn}(h^{-\frac{n}{2}},f).
\]
\end{proof}
Now we are ready to prove the main result.\\

\textbf{Proof of Theorem \ref{thm-bsn}}
The proof follows two steps. 

Step $1$: 
Notice that $h^{-\frac{n}{2}}$ is the Green function of the conformal sub-Laplacian $L_{\bsn}-n^2$ with pole $(\pi/2,0)$ (the south pole $-e_n$ of $\bsn$). For any $f\in D_S$ we let
\begin{equation}\label{Lh}
L^hf:= L_{\bsn} f+\frac{2\Gamma_{\bsn}(h^{-\frac{n}{2}},f)}{h^{-\frac{n}{2}}}=\frac{L_{\bsn}(h^{-\frac{n}{2}}f)}{h^{-\frac{n}{2}}}-n^2f.
\end{equation}

Let $X_t^h$ and $X_t$ be Markov processes generated by $\frac{1}{2}L^h$ and $\frac{1}{2}L_{\bsn}$, issued from $x\in\bsn$. We first prove that $X_t^h$ is $X_t$ conditioned to be at the south pole $-e_n$ at time $T$, where $T$ is a random time with distribution \eqref{dis-T}. 
 
It is sufficient to prove that for any $f\in D_S$,
\begin{equation}\label{condition}
\bE_x\edg{f(X_t^h)
}=\bE_x\edg{f(X_t)\mathbbm{1}_{t<T}|X_T=-e_n}
\end{equation}
Let $P_t^h$ and $P_t$ be the heat semigroups generated by $L^h$ and $L_{\bsn}$ respectively, then by iterating \eqref{Lh} it is not hard to obtain for any $x\in\bsn$,
\[
P_t^h(f(x))=h(x)^{\frac{n}{2}}e^{-tn^2}P_t(h^{-\frac{n}{2}}(x)f(x)),
\]
that is
\[
\bE_x\edg{f(X_t^h)
}=\frac{1}{h^{-\frac{n}{2}}(x)}e^{-tn^2}\bE_x\edg{h^{-\frac{n}{2}}(X_t)f(X_t)
}
=\bE_x\edg{\frac{e^{-tn^2}h^{-\frac{n}{2}}(X_t)}{h^{-\frac{n}{2}}(x)}f(X_t)
}.
\]
Proving \eqref{condition} is then equivalent to proving
\begin{equation}\label{eq-equivalent}
 \bE_x\edg{f(X_t)\mathbbm{1}_{t<T}|X_{T}=-e_n}=\bE_x\edg{\frac{e^{-tn^2}h^{-\frac{n}{2}}(X_t)}{h^{-\frac{n}{2}}(x)}f(X_t)
 }.
\end{equation}
Note that
\[
 \bE_x\edg{f(X_t)\mathbbm{1}_{t<T}|X_{T}=-e_n}=\frac{\bE_x\edg{f(X_t)\mathbbm{1}_{t<T}\mathbbm{1}_{X_T=-e_n}}}{\bE_x\edg{X_T=-e_n}}.
\]
Assume $T$ is an exponential random variable with parameter $-n^2$ under the original probability measure, we have
\[
\bE_x\edg{f(X_t)\mathbbm{1}_{t<T}\mathbbm{1}_{X_T=-e_n}}=\bE_x\edg{e^{-tn^2}h^{-\frac{n}{2}}(X_t)f(X_t)
}
\] 
and 
\[
\bE_x\edg{X_T=-e_n}=\int_0^{+\infty}p_t(x,-e_n)e^{-n^2t}dt=h^{-\frac{n}{2}}(x).
\]
Thus \eqref{eq-equivalent} holds when $T$ is an exponential random variable under the original probability measure. Switching to the conditioned probability measure, $T$ then has the distribution
\[
\mathbb{P}^h_x\edg{T>t}=e^{-n^2t}\frac{\bE_x\edg{h^{-\frac{n}{2}}(X_t)}}{h^{-\frac{n}{2}}(x)}=\frac{\int_t^{+\infty}e^{-n^2s}p_s(-e_n,x)ds}{\int_0^{+\infty}e^{-n^2t}p_t(-e_n,x)dt}.
\]

Step $2$: 
Next we prove the time change.  Let $Y_t$ be the Markov process  generated by $\frac{1}{2}L_{\fH}$ and issued from $\mathcal{C}_1^{-1}(x)$, we claim that $Y_t$ is mapped by Cayley transform to a time-changed version of $X^h$, i.e.,
\begin{equation}\label{eq-X-h-A}
X^h_{\mathcal{A}_t}=\mathcal{C}_1(Y_t)
\end{equation}
where the time change is given by $\mathcal{A}_t=\int_0^tH(Y_s)^{-1}ds$.  To see this, we consider for any $F=f\circ \mathcal{C}_1\in D_H$, the associated martingale $M_t^F$ that is given by
\[
M_t^F=F(Y_t)-\frac{1}{2}\int_0^t L_{\fH}F(Y_s)ds.
\]
By plugging in \eqref{eq-X-h-A}, \eqref{Doob} and  \eqref{Lh} we have
\[
M_t^F=f(X^h_{\mathcal{A}_t})-\frac{1}{2}\int_0^t  (L_{\fH}F)\circ\mathcal{C}_1^{-1}(X^h_{\mathcal{A}_s})ds=f(X^h_{\mathcal{A}_t})-\frac{1}{2}\int_0^{t} H(Y_s)L^hf(X^h_{\mathcal{A}_s})ds.
\]
Let $\sigma_t$ be the hitting time such that $\sigma_t=\inf\{u, \mathcal{A}_u>t \}$, then clearly  $\mathcal{A}_{\sigma_t}=t=\sigma_{\mathcal{A}_t}$. By changing variable  $s=\sigma_u$ we obtain
\[
M_t^F=f(X^h_{\mathcal{A}_t})-\frac{1}{2}\int_0^{\sigma_{\mathcal{A}_t}} H(Y_s)L^hf(X^h_{\mathcal{A}_s})ds=f(X^h_{t})-\frac{1}{2}\int_0^{\mathcal{A}_t}  H(Y_{\sigma_u})L^hf(X^h_u)\sigma_u'du.
\]
Note for any $u>0$ we have $u=\mathcal{A}_{\sigma_u}=\int_0^{\sigma_u}H(Y_s)ds$. This implies that 
\[
1=H(Y_{\sigma_u})\sigma_u'.
\]
Therefore 
\[
M_t^F=f(X^h_{t})-\frac{1}{2}\int_0^{\mathcal{A}_t} L^hf(X^h_u)du,
\]
and it completes the proof.

\section{Inversion of Brownian motions on Heisenberg group}
In this section we consider the inversion of Brownian motion on Heisenberg group. First we construct the inverse map by composing two Cayley transforms $\mathcal{C}_1$ and $\mathcal{C}_2$, between $\fH$ and $\bsn$ minus a point ($-e_n$ and $e_n$ respectively). We have already discussed  $\mathcal{C}_1$ in the previous section. Now let us consider $ \mathcal{C}_2:\fH\to \bsn \setminus\{ e_n\}$ where $e_{n}$ is the north pole on $\bsn$. We have
\[
\mathcal{C}_2: \left( z,t\right)\to\left(\frac{2z_1}{1+|z|^2+2it},\dots,  \frac{2z_n}{1+|z|^2+2it}, -\frac{1-|z|^2-2it}{1+|z|^2+2it}\right).
\]
and
\[
\mathcal{C}_2^{-1}: \lbrace \zeta_1,\cdots \zeta_{n+1}\rbrace\to \bigg\{ \frac{\zeta_1}{1-\zeta_{n+1}},\dots, \frac{i}{2}\frac{\overline{\zeta_{n+1}}-\zeta_{n+1}}{|1-\zeta_{n+1}|^2}\bigg\}. 
\]
Let $\mathcal{K}:\fH\setminus\{0\}\longrightarrow\fH\setminus\{0 \}$ be such that $\mathcal{K}=\mathcal{C}_2^{-1}\circ \mathcal{C}_1$, then
\[
\mathcal{K}:(z_1,\cdots z_n,t)\rightarrow\left(\frac{z_1}{|z|^2-2it},\dots,\frac{z_n}{|z|^2-2it},\frac{t}{|z|^4+4t^2} \right).
\]
Clearly $\mathcal{K}$ is an involution on $\fH\setminus \{0\}$ and preserve the Kor\'anyi ball $\{(z,t)\in\fH,|z|^4+4t^2=1\}$. Indeed it is the Kelvin transform generalized to Heisenberg group (see \cite{Kr}).
 
 
For any $(r_H, t)\in \R_{\ge0}\times \R$ and $(r_S,\theta)\in [0,\frac{\pi}{2})\times \R/2\pi\mathbb{Z}$, we let $\tilde{h}(r_S, \theta)=1+\cos^2r_S-2\cos r_S\cos\theta$ and $\tilde{H}(r_H,t)=\frac{4(r_H^4+4t^2)}{(1+r_H^2)+4t^2}$, then $\mathcal{K}^*H=(\mathcal{C}_2\circ\mathcal{C}_1^{-1})^*H=\tilde{H}$. Moreover, simple calculations show that 
\[
\tilde{h}=(\mathcal{C}_2^{-1})^*H, \quad h=(\mathcal{C}_2^{-1})^{*}\tilde{H}, \quad \tilde{h}=(\mathcal{C}_1^{-1})^*\tilde{H}.
\]
Let $N(r_H,t)=r_H^4+4t^2$. By comparing the conformal Laplacians induced by $\mathcal{C}_1$ and $\mathcal{C}_2$, we obtain the following relation.
\begin{theorem}
For any function $F\in D_H$,
\[
(\mathcal{K}_*L_{\fH}) F={N}^{\frac{n}{2}+1}L_{\fH}(N^{-n/2}F) .
\]
\end{theorem}
\begin{proof}
First we notice that for all $f\in D_S$,
\[
\tilde{h}^{\frac{n}{2}+1} (-L_{\bsn}+n^2) \left( \tilde{h}^{-\frac{n}{2}}f\right)=-(\mathcal{C}_{2*}L_{\fH})f.
\]
Together with \eqref{LH-LS} we obtain
\[
{h}^{-(\frac{n}{2}+1)}(\mathcal{C}_{1*}L_{\fH})( h^{\frac{n}{2}}f)=\tilde{h}^{-(\frac{n}{2}+1)}(\mathcal{C}_{2*}L_{\fH})( \tilde{h}^{\frac{n}{2}}f).
\]
Thus 
\[
(\mathcal{C}_{1*}L_{\fH})f=n^{-(\frac{n}{2}+1)}(\mathcal{C}_{2*}L_{\fH})\left(n^{\frac{n}{2}}f\right),
\]
where $n=\frac{\tilde{h}}{h}$. Note that $(\mathcal{C}_2)^*n=N^{-1}$, we have for any $F=\mathcal{C}_2^*f$,
\[
(\mathcal{K}_*L_{\fH})F=N^{\frac{n}{2}+1}L_{\fH}(N^{-\frac{n}{2}}F).
\]
\end{proof}

 Now we are ready to prove the relation between the inversion of Brownian motion on $\fH$ and the time changed Brownian bridge on $\fH$.\\
 
\textbf{Proof of Theorem \ref{thm-K}}
Note that $N^{-\frac{n}{2}}$ is the Green function of the sub-Laplacian $L_{\fH}$ with pole  $(0,0)$. We let
\begin{equation}\label{LN}
L^NF:=L_{\fH}F+2N^{\frac{n}{2}}\Gamma_{\fH}(N^{-\frac{n}{2}},F),
\end{equation}
where $\Gamma_{\fH}(F,G)=\frac{1}{2}(L_{\fH}(FG)-fL_{\fH}G-GL_{\fH}F)$ for any $F, G\in D_H$. From the previous theorem we have
\[
\mathcal{K}_*L_{\fH}=NL^N.
\]

Let $X_t^N$ and $X_t$ be Markov processes generated by $\frac{1}{2}L^N$ and $\frac{1}{2}L_{\fH}$. We first prove that $X_t^N$ is $X_t$ conditioned to be at the origin. 
 
It suffices to prove that for any $F\in D_H$,
\begin{equation}\label{condition-N}
\bE_x\edg{F(X_t^N)
}=\bE_x\edg{F(X_t)\mathbbm{1}_{t<T}|X_\infty=(0,0)}
\end{equation}
Let $P_t^N$ and $P_t$ be the heat semigroups generated by $L^N$ and $L_{\fH}$ respectively, then by iterating \eqref{LN} it is not hard to obtain
\[
P_t^N(F(x))=N(x)^{-\frac{n}{2}}P_t(N(x)^{-\frac{n}{2}}F(x)),
\]
that is
\[
\bE_x\edg{F(X_t^N)
}=\frac{1}{N(x)^{-\frac{n}{2}}}\bE_x\edg{N(X_t)^{-\frac{n}{2}}F(X_t)
}
=\bE_x\edg{\frac{N(X_t)^{-\frac{n}{2}}}{N(x)^{-\frac{n}{2}}}F(X_t)
}.
\]
From \eqref{condition-N}, we just need to show that
\[
 \bE_x\edg{F(X_t)|X_{\infty}=0}=\bE_x\edg{\frac{N(X_t)^{-\frac{n}{2}}}{N^{-\frac{n}{2}}(x)}F(X_t)}.
\]
This is an easy consequence of $\bE_x\edg{X_\infty=0}=N^{-\frac{n}{2}}(x)$ and 
\[
\bE_x\edg{F(X_t)\mathbbm{1}_{X_{\infty}=0}}=\bE_x\edg{F(X_t)\bE_x[\mathbbm{1}_{X_{\infty}=0}|\mathcal{F}_t]}=\bE_x\edg{{N(X_t)^{-\frac{n}{2}}}F(X_t)}.
\]
Next we prove the time change.  Consider the Markov process  generated by $\frac{1}{2}\mathcal{K}_*(L_{\fH})$. It is the image of $X_t$ under Kelvin transform, namely $\mathcal{K}(X_t)$. We claim 
\begin{equation}\label{eq-X-N-A}
\mathcal{K}(X_t)=X^N_{\mathcal{A}_t}
\end{equation}
where $\mathcal{A}_t=\int_0^tN(X_s)^{}ds$ is the time-change of $X^N$.  For any $F\in D_H$, we consider the associated martingale 
\[
M_t^F:=F(X_t)-\frac{1}{2}\int_0^t L_{\fH}F(X_s)ds.
\]
Denote $\tilde{F}=(\mathcal{K}^{})^*F$. By plugging in \eqref{eq-X-N-A}, we obtain
\[
M_t^F=\tilde{F}(X^N_{\mathcal{A}_t})-\frac{1}{2}\int_0^t  (L_{\fH}F)\circ\mathcal{K}^{-1}(X^h_{\mathcal{A}_s})ds=\tilde{F}(X^N_{\mathcal{A}_t})-\frac{1}{2}\int_0^{t} N(X_s)L^N\tilde{F}(X^N_{\mathcal{A}_s})ds.
\]
Let $\sigma_t$ be the hitting time such that $\sigma_t=\inf\{u, \mathcal{A}_u>t \}$, then clearly  $\mathcal{A}_{\sigma_t}=t=\sigma_{\mathcal{A}_t}$. By changing variable  $s=\sigma_u$ we have
\[
M_t^F=\tilde{F}(X^N_{\mathcal{A}_t})-\frac{1}{2}\int_0^{\sigma_{\mathcal{A}_t}} N(X_s)L^N\tilde{F}(X^N_{\mathcal{A}_s})ds=\tilde{F}(X^N_{t})-\frac{1}{2}\int_0^{\mathcal{A}_t}  N(X_{\sigma_u})L^N\tilde{F}(X^N_u)\sigma_u'du.
\]
Note for any $u>0$ we have $u=\mathcal{A}_{\sigma_u}=\int_0^{\sigma_u}N(X_s)ds$. By differentiating both sides with respect to $u$ we obtain
\[
1=N(X_{\sigma_u})\sigma_u'.
\]
Hence 
\[
M_t^F=\tilde{F}(X^N_{t})-\frac{1}{2}\int_0^{\mathcal{A}_t} L^N\tilde{F}(X^N_u)du,
\]
and we have the conclusion.

\end{document}